\documentclass[oneside,english]{amsart}
\usepackage[T1]{fontenc}
\usepackage[latin9]{inputenc}
\usepackage{geometry}
\usepackage{amsthm}
\usepackage{amstext}
\usepackage{amssymb}
\usepackage{esint}

\makeatletter
\numberwithin{equation}{section}
\numberwithin{figure}{section}
\theoremstyle{plain}
\newtheorem{thm}{\protect\theoremname}
  \theoremstyle{definition}
  \newtheorem{defn}[thm]{\protect\definitionname}
  \theoremstyle{remark}
  \newtheorem{rem}[thm]{\protect\remarkname}
  \theoremstyle{plain}
  \newtheorem{prop}[thm]{\protect\propositionname}

\makeatother

\usepackage{babel}
  \providecommand{\definitionname}{Definition}
  \providecommand{\propositionname}{Proposition}
  \providecommand{\remarkname}{Remark}
\providecommand{\theoremname}{Theorem}

\begin{document}
\global\long\def\Re{\textrm{Re}}

\global\long\def\Im{\textrm{Im}}

\global\long\def\bbP{\mathbb{P}}

\global\long\def\O{\mathrm{\mathtt{\mathsf{\mathbf{\mathrm{O}}}}}}

\global\long\def\o{\mathrm{\mathtt{\mathrm{o}}}}

\global\long\def\DD{\mathcal{D}}

\global\long\def\deq{\overset{d}{=}}

\global\long\def\bbN{\mathbb{N}}

\global\long\def\bbQ{\mathbb{Q}}

\global\long\def\bbR{\mathbb{R}}

\global\long\def\bbZ{\mathbb{Z}}

\global\long\def\bbC{\mathbb{C}}

\global\long\def\eset{\emptyset}

\global\long\def\nto{\nrightarrow}

\global\long\def\re{\mathrm{Re\,}}

\global\long\def\im{\mathrm{Im\,}}

\global\long\def\calB{\mathcal{B}}

\global\long\def\limti{{\displaystyle \lim_{n\to\infty}}}

\global\long\def\sumnti{{\displaystyle \sum_{n=1}^{\infty}}}

\global\long\def\sumktn{{\displaystyle \sum_{k=1}^{n}}}

\global\long\def\E{{\bf E}}

\global\long\def\ind{{\bf 1}}

\global\long\def\One{{\bf 1 }}

\global\long\def\ZZ{\mathbb{Z}}

\global\long\def\NN{\mathbb{N}}

\global\long\def\RR{\mathbb{R}}

\global\long\def\PP{\mathbb{P}}

\global\long\def\BB{\mathcal{B}}

\global\long\def\O{\mathrm{\mathtt{\mathsf{\mathbf{\mathrm{O}}}}}}

\global\long\def\o{\mathrm{\mathtt{\mathrm{o}}}}

\global\long\def\DD{\mathcal{D}}

\global\long\def\deq{\overset{d}{=}}

\title{Weak invariance principle for the local times of Gibbs-Markov processes}

\author{Michael Bromberg\\
School of Mathematical Sciences, Tel Aviv University. Tel Aviv 69978,
Israel.}
\begin{abstract}
The subject of this paper is to prove a functional weak invariance
principle for the local time of a process generated by a Gibbs-Markov
map. More precisely, let $\left(X,\calB,m,T,\alpha\right)$ is a mixing,
probability preserving Gibbs-Markov{\normalsize{}. and let $\varphi\in L^{2}\left(m\right)$
be an aperiodic function with mean $0$. Set $S_{n}=\sum_{k=0}^{n}X_{k}$
and define the hitting time process $L_{n}\left(x\right)$ be the
number of times $S_{k}$ hits $x\in\ZZ$ up to step $n.$ The normalized
local time process $l_{n}\left(x\right)$ is defined by 
\[
l_{n}\left(t\right)=\frac{L_{n}\left(\left\lfloor \sqrt{n}x\right\rfloor \right)}{\sqrt{n}},\,\, x\in\RR.
\]
 We prove under that $l_{n}\left(x\right)$ converges in distribution
to the local time of the Brownian Motion. The proof also applies to
the more classical setting of local times derived from a subshift
of finite type endowed with a Gibbs measure. }{\normalsize \par}
\end{abstract}
\maketitle

\section{Introduction}

Let $\left(X,\calB,m,T,\alpha\right)$ be a mixing, probability preserving
Gibbs-Markov map on a standard probability space. Let $\varphi\in L^{2}\left(m\right)$
be an integer valued function with mean $0$. We assume that $\varphi$
is a uniformly Lipchitz continuous function on the partition $\beta=T\alpha$,
i.e. $D_{\beta}f:=\sup_{a\in\beta}D_{a}f<\infty$, where $D_{a}f=\sup_{x,y\in a}\frac{\left|f\left(x\right)-f\left(y\right)\right|}{d\left(x,y\right)}$
is the Lipchitz norm on $a$ and $d\left(\cdot,\cdot\right)$ is the
complete metric on $X$.

In what follows, convergence in distribution of random variables $X_{n}$
taking values in some standard probability space $\Omega$ to a limit
$X$, means that for every bounded and continuous $f:\Omega\rightarrow\bbR$,
$E\left(f\left(X_{n}\right)\right)\rightarrow E\left(f\left(X\right)\right)$,
where $E\left(\cdot\right)$ denotes expectation. In this case, we
write $X_{n}\overset{d}{\longrightarrow}X$. 

Let $S_{n}\left(x\right):=\sum_{k=0}^{n-1}\varphi\left(T^{k}\left(x\right)\right)$
and $\omega_{n}\left(t\right)=\frac{S_{\left[nt\right]}}{\sqrt{n}}$,
where $\left[x\right]$ is the integral value of $x$, $t\in\left[0,1\right]$.
The central limit theorem for $S_{n}$ states that $\frac{S_{n}}{\sqrt{n}}$
converges in distibution to the Gaussian distribution $\mathcal{N}\left(0,\sigma^{2}\right)$,
where $\sigma^{2}=\lim_{n\rightarrow\infty}\frac{Var\left(S_{n}^{2}\right)}{n}$
is the asymptotic variance of $S_{n}$. The stronger, functional CLT
states that the random functions $\omega_{n}\left(\cdot\right)$ converge
in distribution to $\omega\left(\cdot\right)$, where $\omega\left(\cdot\right)$
is the Brownian motion sastisfying $E\left(\omega\left(t\right)\right)=0$,
$Var\left(\omega\left(t\right)\right)=\sigma^{2}t$. Here, convergence
in distribution is of random variables taking values on the Skorokhod
space $D\left[0,1\right]$ of fucntions on $\left[0,1\right]$ that
are continuous from the right with finite limits on the left (cadlag
functions). 

We wish to establish a distributional invariance principle for the
local time of the sequence $\omega_{n}$. To make this precise, define
the occupation times of a function $f\in D\left[0,1\right]$ by 
\[
\nu_{f}\left(A\right)=\int_{0}^{1}\ind_{A}\left(f\left(t\right)\right)dt,\,\, A\in B\left(\bbR\right).
\]

Recall that the occupation measure of the Brownian motion is almost
surely, absolutely contiuous with respect to the Lebesgue measure
on $\bbR$ \cite{MP}. The (random) density function with respect
to the occupation measure, which we denote by $l\left(\cdot\right)$
is the local time of the Brownian motion. 

We define the local time of $\omega_{n}$ at the point $x$ by 
\[
l_{n}\left(x\right)=\frac{\#\left\{ 0\leq k\leq n:\, S_{k}=\left[\sqrt{n}x\right]\right\} }{\sqrt{n}}.
\]

$l_{n}$ is the normailized number of visits to the point $\left[\sqrt{n}x\right]$
by the process $\left\{ S_{k}\right\} $ up to time $n$. It may be
roughly viewed as the density function of the atomic occupation measure
$\omega_{n}$. In fact (as section ... shows) the differences $\nu_{\omega_{n}}\left[a,b\right]-\int_{\bbR}l_{n}\left(x\right)dx$
converge in distribution to $0$. 

Existence of local time for the Brownian motion ensures that $\nu_{\omega_{n}}\left(A\right)\overset{d}{\longrightarrow}\nu_{\omega}\left(A\right)$
for every $A\in\mathcal{B}\left(\bbR\right)$ with boundary of Lebesgue
measure $0$. We wish to establish the convergence in distribution
of the corresponding local times. 

Since the local time of the Brownian motion is an a.s continuous function,
we may consider $l_{n}$ and $l$ as a family of random variables
taking values in the space $D$ of cadlag functions on $\bbR$ (see
\cite{Bil}). 
\begin{thm}
\label{thm:Main Theorem}Let $\varphi\in L^{2}\left(m\right)$ with
$\sup_{a\in\beta}D_{a}\varphi<\infty$, $m\left(\varphi\right)=0$.
If $\varphi$ is aperiodic (see definition \ref{def: Aperiodicity})
, then $l_{n}\left(\cdot\right)\overset{d}{\longrightarrow}l\left(\cdot\right)$. 
\end{thm}
To prove the theorem, we prove tightness of the sequence $l_{n}$
in section \ref{sec:Tightness} and then identify $l$ as the only
possible limit point for $l_{n}$ in section \ref{Sec:Identifying the Only Possible Limit Point}.

\section{Characteristic Function Operators}

Throughout this section, let $\left(X,\calB,m,T,\alpha\right)$ be
a mixing, probability preserving Gibbs-Markov map. For a measurable
partition $\beta$ of $X$, denote by $L_{p,\beta}$ the intersection
of $L^{p}$ with the space of all functions with a finite Lipchitz
norm, i.e. $L_{p,\beta}=\left\{ f\in L^{p}\left(m\right)|\, D_{\beta}f<\infty\right\} $
(see introduction for the definition of the Lipchitz norm). Throughout
the rest of this section $\beta$ denotes the partition $T\alpha$. 

Consider $T$ as an operator on $L^{\infty}\left(m\right)$ defined
by $Tf=f\circ T$. Then the transfer operator $\hat{T}:L^{1}\left(m\right)\rightarrow L^{1}\left(m\right)$
is the pre-dual of $T$, uniquely defined by the equation 
\[
\int f\cdot g\circ T\, d\mu=\int\hat{T}f\cdot g\, d\mu\,\,\forall f\in L^{1},g\in L^{\infty}.
\]
Recall that an operator $S$ on a Banach space $B$ is called quasi-compact
if there exist $S$-invariant closed subspaces $F,H$ such that:
\begin{enumerate}
\item $F$ is finite dimensional and $B=F\oplus H$. 
\item $T$ is diagonizable when restricted to $F$ with all eigenvalues
having modulus equal to the spectral radius of $T$, denoted by $\rho\left(T\right)$.
\item When restricted to $H$, the spectral radius of $T$ is strictly less
than $\rho\left(T\right)$.\end{enumerate}
\begin{thm}
\cite{AD} $\hat{T}$ is a quasi-compact operator on the space $L:=L_{\infty,\beta}$.
Moreover, $\hat{T}f=m\left(f\right)+Qf$, where $m\left(f\right)$
is interpreted as a constant function on $X$ and $\rho\left(Q\right)<1$. 
\end{thm}
For a measurable function $\varphi:X\rightarrow\bbR$, the characteristic
function operators $P_{t}:L^{1}\left(m\right)\rightarrow L^{1}\left(m\right)$,
$t\in\bbR$ are defined by 
\begin{equation}
P_{t}f=\hat{T}\left(e^{it\varphi}f\right).\label{eq:Characteristic function operator}
\end{equation}
If $\varphi\in L_{2,\beta}$ then $P_{t}$ is a twice continuously
differentiable function of $t$. 

Restricting $P_{t}$ to act on $L$ and using the implicit function
theorem (see \cite{HeH}) together with the quasi-compactness of $\hat{T}$
on $L$, we may obtain the Taylor's expansion of the operator $P_{t}$
near $0$. In case $m\left(\varphi\right)=0$ (as we assume for our
purposes), in a sufficiently small neighborhood of $0$, $P_{t}:L\rightarrow L$
is of the form 
\[
P_{t}=\lambda_{t}\pi_{t}+N_{t}
\]
where $\lambda_{t}$ is an eigenvalue with absolute value not exceeding
$1$, $\pi_{t}$ is a projection onto a one dimensional vector space
generated by an eigenfunction $v_{t}$ and $\rho\left(N_{t}\right)<q<1$
for some constant $q$. Moreover, $v_{t}$, $\pi_{t}$, $\lambda_{t}$
are twice continuously differentiable functions of $t$ and the Talyor's
expansions for $\lambda_{t}$ and $\pi_{_{t}}$ are given by 
\begin{eqnarray}
\lambda_{t} & = & 1-\sigma t^{2}+\o\left(t^{2}\right)\nonumber \\
\pi_{t} & = & m+\eta_{t}\label{eq: Taylor's decomposition for characteristic fn op}\\
v_{t} & = & \ind+\O\left(t\right)
\end{eqnarray}
where $\left\Vert \eta_{t}\right\Vert =\O\left(t\right)$.

Dividing the eigenfunctions $v_{t}$ by $m\left(v_{t}\right)$ which
do not vanish in a neighborhood of $t$ (and multiplying $\pi_{t}$
by the same value) we may assume that $m\left(v_{t}\right)=1$. 

We also need the fact that $v'_{0}$ is a purely imaginary function.
To see this note that the equality 
\[
P_{t}v_{t}=\lambda_{t}v_{t}
\]
implies
\begin{eqnarray*}
P'_{0}v_{0}+P_{0}v'_{0} & = & \lambda'_{0}v_{0}+\lambda_{0}v'_{0}\\
 & = & v'_{0}
\end{eqnarray*}
Since $v_{0}=\ind$ we obtain 
\[
P'_{0}\ind=\left(I-P_{0}\right)v'_{0}
\]

Now, $P'_{0}\left(\ind\right)=\hat{T}\left(i\varphi\right)$ is purely
imaginary, since $\hat{T}f$ is real if $f$ is real. Moreover, $m\left(\varphi\right)=0$
implies $m\left(P'_{0}\ind\right)=0$. By corollary 3.6 in \cite{HeH},
the equation $P'_{0}\ind=\left(I-\hat{T}\right)f$, $m\left(f\right)=0$
has a unique solution. Since $v'_{0}$ is the solution to this equation
($m\left(v_{t}\right)\equiv1\implies m\left(v'_{t}\right)\equiv0$)
it follows that $v'_{0}$ is purely imaginary. 

In what follows, we restrict $P_{t}$ to act on $L$. 
\begin{defn}
\label{def: Aperiodicity}A measurable function $\varphi:X\rightarrow\bbZ$
is aperiodic if there is no non-trivial character $\gamma\in\hat{\bbZ}$,
such that $\gamma\circ\varphi$ is $T$-cohomologous to a constant,
i.e. the only solution to the equation 
\[
e^{it\varphi}=\lambda\frac{f\circ T}{f}
\]
with $f:X\rightarrow\mathbb{T}$ measurable, is $t\in2\pi\bbZ$, $f\equiv1$,
$\lambda=1$, $f\equiv1$. $\varphi$ is periodic if it is not aperiodic. \end{defn}
\begin{rem}
\label{Spectral radius of characteristic function operator}If $\varphi$
is aperiodic then the characteristic function operator $P_{t}$ defined
by \ref{eq:Characteristic function operator} has spectral radius
strictly less than $1$ for all $t\notin2\pi\bbZ$. By continuity
of $P_{t}$, this implies that in every compact set $K\subseteq\bbR\setminus2\pi\bbZ$,
there exists a constant $q_{K}<1$, such that $\left\Vert P_{t}^{n}\right\Vert \leq q_{K}^{n}$
for all sufficiently large $n$. 
\end{rem}

\section{Probability Estimates}

Throughout this section we assume that the conditions of theorem \ref{thm:Main Theorem}
hold (hence, all results of the previous section also hold). 
\begin{prop}
\label{prop: Local limi prop}There exists a constant $C$ such that
for any $x\in\bbZ$, $\sqrt{n}\cdot m\left(S_{n}=x\right)<C$.\end{prop}
\begin{proof}
By the inversion formula for Fourier transform and definition of the
characterisitic function operator,
\begin{eqnarray*}
m\left(S_{n}=x\right) & = & \Re\int_{\left[-\pi,\pi\right]}m\left(e^{itS_{n}}\right)e^{-itx}dt\\
 & = & \Re\int_{\left[-\pi,\pi\right]}m\left(P_{t}^{n}\ind\right)e^{-itx}dt.
\end{eqnarray*}
By (\ref{eq: Taylor's decomposition for characteristic fn op}) there
exist a $0<\delta<\pi$ such that $P_{t}=\lambda_{t}\pi_{t}+N_{t}$
where $\lambda_{t}=1-\sigma t^{2}+\epsilon\left(t\right)$, where
$\left|\epsilon\left(t\right)\right|\leq\epsilon t^{2}$ for some
$\epsilon$ satisfying $c:=\sigma-\epsilon>0$, and the spectral radius
of $N_{t}$ satisfies $\rho\left(N_{t}\right)\leq q<1$ for all $t\in\left(-\delta,\delta\right)$.
Write $C_{\delta}=\left(-\delta,\delta\right)$ and $\bar{C}_{\delta}=\left[-\pi,\pi\right]\setminus\left(-\delta,\delta\right)$.
Then 
\begin{eqnarray}
\Re\int_{\left[-\pi,\pi\right]}m\left(P_{t}^{n}\ind\right)e^{-itx}dt & \leq & \int_{C_{\delta}}\left\Vert P_{t}^{n}\ind\right\Vert _{\mathcal{L}}dt+\int_{\bar{C}_{\delta}}\left\Vert P_{t}^{n}\ind\right\Vert _{\mathcal{L}}dt.\label{eq:Inner Eq1}
\end{eqnarray}
Now, by remark \ref{Spectral radius of characteristic function operator}
$\sup_{t\in\bar{C}_{\delta}}\left\Vert P_{t}^{n}\ind\right\Vert _{\mathcal{L}}$
exponentially tends to $0$. Hence, the second term on the right side
of the above inequality multiplied by $\sqrt{n}$ tends to $0$ as
$n$ tends to $\infty$ and in particular, is uniformly bounded. To
bound the first term, write 
\[
\left\Vert P_{t}^{n}\ind\right\Vert _{\mathcal{L}}\leq\left|\lambda_{t}^{n}\right|+\left\Vert N_{t}^{n}\ind\right\Vert \leq\left(1-ct^{2}\right)^{n}+\tilde{c}q^{n}
\]
for some constant $\tilde{c}$, which exists since $\rho\left(N_{t}\right)\leq q$
on $C_{\delta}$ . Then $ $
\[
\int_{C_{\delta}}\left\Vert P_{t}^{n}\ind\right\Vert _{\mathcal{L}}dt\leq\int_{C_{\delta}}\left(1-ct^{2}\right)^{n}dt+2\delta\tilde{c}q^{n}
\]
and by applying the substitution $t=\frac{y}{\sqrt{n}}$ we obtain
\[
\int_{C_{\delta}}\left(1-ct^{2}\right)^{n}dt=\frac{1}{\sqrt{n}}\int\limits _{\left(-\sqrt{n}\delta,\sqrt{n}\delta\right)}\left(1-c\frac{y^{2}}{n}\right)^{n}dt\leq\frac{1}{\sqrt{n}}\int\limits _{-\infty}^{\infty}e^{-cy^{2}}dy
\]
Since, the last integral converges, $\sqrt{n}\int_{C_{\delta}}\left(1-ct^{2}\right)^{n}dt$
is uniformly bounded by a constant. Since, the second term on the
right hand side of the inequality \ref{eq:Inner Eq1} tends to $0$
exponentially fast, this completes the proof.\end{proof}
\begin{rem}
\label{Rem: LLT}Note that during the proof, we showed that $\sqrt{n}\int_{C_{\delta}}\left|\lambda_{t}\right|^{n}dt$
is uniformly bounded by a constant. Essentially the same proof may
be used to show that $n\int_{C_{\delta}}\left|t\lambda_{t}^{n}\right|dt$
is uniformly bounded by a constant. We use both these facts in the
proof of the next proposition. \end{rem}
\begin{prop}
\label{prop: Potential Kernel}For all $x,y\in\bbZ$, there exists
a constant $C$ such that $\sum_{n=1}^{\infty}\left|m\left(S_{n}=x\right)-m\left(S_{n}=y\right)\right|\leq C\left|x-y\right|$. \end{prop}
\begin{proof}
By the inversion formula, 
\[
\left|m\left(S_{n}=x\right)-m\left(S_{n}=y\right)\right|=\left|\Re\int_{-\left[\pi,\pi\right]}m\left(P_{t}^{n}\ind\right)\left(e^{itx}-e^{ity}\right)dt\right|.
\]
By (\ref{eq: Taylor's decomposition for characteristic fn op}) there
exist a $0<\delta<\pi$ such that $P_{t}=\lambda_{t}\pi_{t}+N_{t}$
where $\left|\lambda_{t}\right|\leq1-ct^{2}$ for some positive constant,
the spectral radius of $N_{t}$ satisfies $\rho\left(N_{t}\right)\leq q<1$
for all $t\in\left(-\delta,\delta\right)$, and $\pi_{t}=m+\eta_{t}$
with $\left\Vert \eta_{t}\right\Vert \leq\tilde{c}t$ for some $\tilde{c}\geq0$.
Write $C_{\delta}=\left(-\delta,\delta\right)$ and $\bar{C}_{\delta}=\left[-\pi,\pi\right]\setminus\left(-\delta,\delta\right)$.
Then 
\[
\left|\Re\int_{-\left[\pi,\pi\right]}m\left(P_{t}^{n}\ind\right)\left(e^{itx}-e^{ity}\right)dt\right|\leq\left|\Re\int_{C_{\delta}}m\left(P_{t}^{n}\ind\right)\left(e^{itx}-e^{ity}\right)dt\right|+\left|\int_{\bar{C}_{\delta}}m\left(P_{t}^{n}\ind\right)dt\right|.
\]
As in the proof of proposition \ref{prop: Local limi prop} the second
term on right side of the above inequality tends to $0$ exponentially
fast and therefore, its sum over $n$ converges. Thus, it is sufficient
to bound the first term. Use the expansion of the characteristic function
operator to get

\[
\left|\Re\int_{C_{\delta}}m\left(P_{t}^{n}\ind\right)\left(e^{itx}-e^{ity}\right)dt\right|\leq\left|\Re\int_{C_{\delta}}\lambda_{t}^{n}m\left(\pi_{t}\ind\right)\left(e^{itx}-e^{ity}\right)dt\right|+\left|2\cdot\int_{C_{\delta}}\left\Vert N_{t}\right\Vert ^{n}dt\right|.
\]
Since $\rho\left(N_{t}\right)\leq q<1$, the sum over $n$ of the
second term on the right hand side is finite. We turn to analyze the
first term. 

\begin{eqnarray}
\left|\Re\int_{C_{\delta}}\lambda_{t}^{n}m\left(\pi_{t}\ind\right)\cdot\left(e^{itx}-e^{ity}\right)dt\right| & = & \left|\int_{C_{\delta}}\left(\Re\lambda_{t}^{n}m\left(\pi_{t}\ind\right)\right)\left(\cos tx-\cos ty\right)dt\right|\nonumber \\
 &  & +\left|\int_{C_{\delta}}\left(\Im\lambda_{t}^{n}m\left(\pi_{t}\ind\right)\right)\left(\sin tx-\sin ty\right)dt\right|\label{eq:Potential kernel to estimate}
\end{eqnarray}
Since $\left|\Re\lambda_{t}^{n}\right|\leq\left|\lambda_{t}^{n}\right|\leq1-ct^{2}$
, and $\left\Vert \pi_{t}\right\Vert _{\mathcal{L}}=1$, we have $ $
\begin{eqnarray*}
\sum_{n=1}^{\infty}\left|\int_{C_{\delta}}\left(\Re\lambda_{t}^{n}m\left(\pi_{t}\ind\right)\right)\left(\cos tx-\cos ty\right)dt\right| & \leq & \sum_{n=1}^{\infty}\int_{c_{\delta}}\left(1-ct^{2}\right)^{n}\left|\cos tx-\cos ty\right|dt\\
 & = & \int_{c_{\delta}}\frac{1}{ct^{2}}\left|\cos tx-\cos ty\right|dt\\
 & \leq & C_{1}\left|x-y\right|
\end{eqnarray*}
for some constant $C_{1}$. 

Estimating the sum over the second term in \ref{eq:Potential kernel to estimate}
is more difficult since $\sin tx-\sin ty$ is of order $t$ instead
of $t^{2}$. We start by using $\pi_{t}=m+\eta_{t}$ to obtain
\begin{equation}
\left|\int_{C_{\delta}}\left(\Im\lambda_{t}^{n}m\left(\pi_{t}\ind\right)\right)\left(\sin tx-\sin ty\right)dt\right|\leq\int_{C_{\delta}}\left|\Im\lambda_{t}^{n}\right|\left|\sin tx-\sin ty\right|dt+\mbox{\ensuremath{\int}}_{C_{\delta}}\left|\lambda_{t}^{n}\tilde{c}t\right|\left|\sin tx-\sin ty\right|dt.\label{eq:Potential kernel to estimate 2}
\end{equation}
Using $\left|\lambda_{t}\right|\leq1-ct^{2}$ we can estimate the
second term on the right hand side of the above inequality. 
\[
\sum_{n=1}^{\infty}\mbox{\ensuremath{\int}}_{C_{\delta}}\left|\lambda_{t}^{n}\tilde{c}t\right|\left|\sin tx-\sin ty\right|dt\leq\tilde{c}\int_{C_{\delta}}\frac{1}{ct}\left|\sin tx-\sin ty\right|dt\leq C_{2}\left|x-y\right|.
\]
The estimatison of the first term on the right hand side of \ref{eq:Potential kernel to estimate 2}
will take up the rest of the proof. 

We first note that $\left|\Im\lambda_{t}^{n}\right|\leq n\left|\lambda_{t}\right|^{n-1}\left|\Im\lambda_{t}\right|$.
Then 
\begin{eqnarray*}
\left|\Im\lambda_{t}\right| & = & \left|m\left(\Im P_{t}v_{t}\right)\right|\\
 & \leq & \left|m\left(\Im P_{t}\ind\right)\right|+\left|m\left(\Im P_{t}\psi_{t}\right)\right|,
\end{eqnarray*}
where $\psi_{t}=\ind-v_{t}$. By definition of the characteristic
function operator, and the fact the $\hat{T}f$ is real if $f$ is
real, 
\begin{eqnarray*}
\left|m\left(\Im P_{t}\psi_{t}\right)\right| & \leq & \left|m\left(\hat{T}\left(\cos t\varphi\Im\psi_{t}\right)\right)\right|+\left|m\left(\hat{T}\left(\sin t\varphi\Re\psi_{t}\right)\right)\right|.
\end{eqnarray*}
 Since $m\left(\psi_{t}\right)=0$ , $m\circ\hat{T}=m$, $\left|1-\cos t\varphi\right|\leq t^{2}\varphi^{2}$,
$\left|\psi_{t}\right|=\o\left(\left|t\right|\right)$ and by the
positivitiy of the transfer operator,
\begin{eqnarray*}
\left|m\left(\hat{T}\left(\cos t\varphi\Im\psi_{t}\right)\right)\right| & = & \left|m\left(\left(\cos t\varphi-1\right)\Im\psi_{t}\right)\right|\\
 & \leq & m\left(t^{2}\varphi^{2}\left|\Im\psi_{t}\right|\right)\\
 & \leq & C_{3}\left|t\right|^{3}
\end{eqnarray*}
where we have used the finiteness of the second moment of $\varphi.$ 

Since $\psi_{0}=0$, $\Re\psi_{0}'=0$ and $\psi_{t}$ is twice continuously
differentiable, 
\[
\left|m\left(\hat{T}\left(\sin t\varphi\Re\psi_{t}\right)\right)\right|\leq C_{4}\left|t\right|^{3}.
\]
Therefore, 
\begin{eqnarray*}
\sum_{n=1}^{\infty}\int_{C_{\delta}}n\left|\lambda_{t}\right|^{n-1}\left|m\left(\Im P_{t}\psi_{t}\right)\right|\left|\sin tx-\sin ty\right|dt & \leq & \sum_{n=1}^{\infty}\int_{C_{\delta}}n\left(1-ct^{2}\right)^{n-1}\left(C_{3}+C_{4}\right)\left|t\right|^{3}\left|\sin tx-\sin ty\right|dt\\
 & \leq & \int_{C_{\delta}}\frac{1}{ct^{4}}\left(C_{3}+C_{4}\right)\left|t\right|^{3}\left|\sin tx-\sin ty\right|dt\\
 & \leq & C_{5}\left|x-y\right|
\end{eqnarray*}
Finally, since $m\left(\varphi\right)=0$ and $m\circ\hat{T}=m$ 
\begin{eqnarray*}
\left|m\left(\Im P_{t}\ind\right)\right| & = & \left|m\left(\sin t\varphi\right)\right|\\
 & = & \left|m\left(\sin t\varphi-t\varphi\right)\right|
\end{eqnarray*}
We split the last integral into parts where $\left|t\varphi\right|\leq1$
and $\left|t\varphi\right|>1$ to obtain 
\begin{eqnarray*}
\left|m\left(\Im P_{t}\ind\right)\right| & \leq & \left|m\left(\ind_{\left\{ \left|t\varphi\right|\leq1\right\} }\left(\sin t\varphi-t\varphi\right)\right)\right|+\left|m\left(\ind_{\left\{ \left|t\varphi\right|>1\right\} }\left(\sin t\varphi-t\varphi\right)\right)\right|\\
 & \leq & \left|m\left(\ind_{\left\{ \left|t\varphi\right|\leq1\right\} }\left|t\varphi\right|^{3}\right)\right|+\left|m\left(2\left|t\varphi\right|\ind_{\left\{ \left|t\varphi\right|>1\right\} }\right)\right|
\end{eqnarray*}
Thus, summing over $n$ and again using $\left|\lambda_{t}\right|\leq\left(1-ct^{2}\right)$
we have 
\begin{eqnarray*}
\sum_{n=1}^{\infty}\int_{C_{\delta}}n\left|\lambda_{t}\right|^{n-1}\left|m\left(\Im P_{t}\ind\right)\right|\left|\sin tx-\sin ty\right|dt & \leq & \int_{C_{\delta}}\frac{1}{ct^{4}}m\left(\left|t\varphi\right|^{3}\ind_{\left\{ \left|t\varphi\leq1\right|\right\} }\right)\left|\sin tx-\sin ty\right|dt+\\
 &  & +2\int_{C_{\delta}}\frac{1}{ct^{4}}m\left(\left|t\varphi\right|\ind_{\left\{ \left|t\varphi\right|>1\right\} }\right)\left|\sin tx-\sin ty\right|dt
\end{eqnarray*}
Bounding $\left|\sin tx-\sin ty\right|$ by $\left|t\left(x-y\right)\right|$
and changing the order of integration in the first term gives 
\begin{eqnarray}
\int_{C_{\delta}}\frac{1}{ct^{4}}m\left(\left|t\varphi\right|^{3}\ind_{\left\{ \left|t\varphi\leq1\right|\right\} }\right)\left|\sin tx-\sin ty\right|dt & \leq & m\left(\left|\varphi\right|^{3}\int_{-\left|\varphi\right|^{-1}}^{\left|\varphi\right|^{-1}}\left|x-y\right|dt\right)\label{eq: 1}\\
 & = & m\left(2\left|\varphi\right|^{2}\right)\left|x-y\right|\nonumber \\
 & \leq & C_{6}\left|x-y\right|\nonumber 
\end{eqnarray}
Changing the order of integration in the second term of \ref{eq: 1}
and using the fact the the integrand is an even function of $t$,
gives 
\begin{eqnarray*}
2\int_{C_{\delta}}\frac{1}{ct^{4}}m\left(\left|t\varphi\right|\ind_{\left\{ \left|t\varphi\right|>1\right\} }\right)\left|\sin tx-\sin ty\right|dt & \leq & 4m\left(\left|\varphi\right|\int_{\left|\varphi\right|^{-1}}^{\delta}\frac{1}{t^{2}}\left|x-y\right|dt\right)\\
 & \leq & C_{7}\left|x-y\right|.
\end{eqnarray*}
This completes the proof.\end{proof}
\begin{prop}
\label{prop:Tightness inequality}For any $\epsilon>0$, $1<\alpha\leq2$
there exists a constant $C$ such that for all $x,y\in\bbR$, $\frac{1}{\sqrt{n}}\leq\left|x-y\right|\leq1$
and $n\in\bbN$, ${\rm P}\left(\left|l_{n}\left(x\right)-l_{n}\left(y\right)\right|>\epsilon\right)\leq C\frac{\left|x-y\right|^{\alpha}}{\epsilon^{6}}$.
\end{prop}
To prove this estimate, let $L_{n}\left(x\right)=\#\left\{ 1\leq k\leq n|\, S_{k}=x\right\} $,
$x\in\bbZ$. Then by definition $l_{n}\left(x\right)=\frac{L_{n}\left(\left[\sqrt{n}x\right]\right)}{\sqrt{n}}$.
It is enough to prove
\begin{prop}
\label{prop: Moment inequality}For any $1<\alpha<2$ there exists
a constant $C$ such that for all $x,y\in\bbZ$ and $n\in\bbN$, $m\left(\left(L_{n}\left(x\right)-L_{n}\left(y\right)\right)^{6}\right)\leq C\cdot\left(\left(\sqrt{n}\left|x-y\right|\right)^{3}+\sqrt{n}^{4}\left|x-y\right|\log n+\sqrt{n}^{4}\left(\log n\right)^{2}\right)$. 
\end{prop}
To see that proposition \ref{prop:Tightness inequality} follows from
this, note that
\begin{eqnarray*}
m\left(\left(l_{n}\left(x\right)-l_{n}\left(y\right)\right)^{6}\right) & = & \frac{1}{\sqrt{n}^{6}}m\left(\left(L_{n}\left(\left[\sqrt{n}x\right]\right)-L_{n}\left(\left[\sqrt{n}y\right]\right)\right)^{6}\right)\\
 & \leq & \frac{1}{n^{3}}C\left(n^{3}\left|x-y\right|^{3}+\left|x-y\right|n^{2}\log n+n^{2}\left(\log n\right)^{2}\right)\\
 & \leq & C\left(\left|x-y\right|^{3}+\frac{1}{n\left|x-y\right|}\left|x-y\right|^{2}\log n+\frac{\left(\log n\right)^{2}}{n\left|x-y\right|^{\alpha}}\left|x-y\right|^{\alpha}\right)\\
 & \leq & \tilde{C}\left|x-y\right|^{\alpha}
\end{eqnarray*}
for any $1<\alpha<2$ ($\tilde{C}$ of course depends on $\alpha$).
The last inequality holds since $\frac{1}{\sqrt{n}}\leq\left|x-y\right|\leq1$.
Proposition \ref{prop:Tightness inequality} now follows from Markov's
inequality.

We turn to the proof of proposition \ref{prop: Moment inequality}.
Using definition of $L_{n}\left(x\right)$ and writing $\psi\mbox{(z)=\ensuremath{\ind}}_{\left\{ x\right\} }-\ind_{\left\{ y\right\} }$,
we obtain

\[
m\left(\left(L_{n}\left(x\right)-L_{n}\left(y\right)\right)^{2p}\right)=m\left(\left(\sum_{k=1}^{n}\ind_{\left\{ x\right\} }\left(S_{k}\right)-\ind_{\left\{ y\right\} }\left(S_{k}\right)\right)^{2p}\right)=\sum_{\bar{i}\in I}m\left(\prod_{l=1}^{2p}\psi\left(S_{i_{l}}\right)\right)
\]
where $I$ is the set of all tuples of length $2p$ of integers between
$1$ and $n$. Clearly, it is enough to prove the estimate for the
case where the coordinates in $I$ are not decreasing. Therefore,
we denote $J=\left\{ \left(j_{1},...,j_{2p}\right)\left|j_{1},..,j_{2p}\in\left\{ 1,...,n\right\} \right.\right\} $
and estimate 
\[
\sum_{\bar{j}\in J}m\left(\prod_{l=1}^{2p}\psi\left(S_{j_{l}}\right)\right).
\]
Fix $\bar{j}\in J$, and let $\bar{k}=\left(j_{1},j_{2}-j_{1},...,j_{2p}-j_{2p-1}\right)$.
Then $ $

\[
\prod_{l=1}^{2p}\psi\left(S_{i_{l}}\right)=\sum_{z_{1},z_{2},,,,,z_{2p}}m\left(\prod_{l=1}^{2p}\psi\left(z_{l}\right)\ind_{\left\{ z_{l}-z_{l-1}\right\} }\left(S_{k_{l}}\right)\right)
\]
where the sum goes over all $\bar{z}=\left\{ \left(z_{1},...z_{2p}\right)\left|z_{i}\in\left\{ x,y\right\} ,i=1,...,2p\right.\right\} $
and $z_{0}=0$. Summing over $z's$ having even subscripts we obtain 

\[
\prod_{l=1}^{2p}\psi\left(S_{i_{l}}\right)=\sum_{z_{1},z_{3},...,z_{2p-1}}m\left(\psi\left(z_{1}\right)\ind_{\left\{ z_{1}\right\} }\left(S_{k_{1}}\right)\left(\prod_{l=1}^{p-1}\psi\left(z_{2l-1}\right)h\left(l,z_{2l-1},z_{2l+1}\right)\right)h\left(p,z_{2p-1}\right)\right)
\]
where 
\[
h\left(l,u,v\right)=\ind_{\left\{ x-u\right\} }\left(S_{k_{2l}}\right)\ind_{\left\{ v-x\right\} }\left(S_{k_{2l}+1}\right)-\ind_{\left\{ y-u\right\} }\left(S_{k_{2l}}\right)\ind_{\left\{ v-y\right\} }\left(S_{k_{2l}+1}\right)
\]
and 
\[
h\left(l,u\right)=\ind_{\left\{ x-u\right\} }\left(S_{2l}\right)-\ind_{\left\{ y-u\right\} }\left(S_{2l}\right).
\]

We can now take absolute values and write

\begin{equation}
\prod_{l=1}^{2p}\psi\left(S_{i_{l}}\right)\leq\sum_{z_{1},z_{3},...,z_{p}}\left|m\left(\ind_{\left\{ z_{1}\right\} }\left(S_{k_{1}}\right)\left(\prod_{l=1}^{p-1}h\left(l,z_{2l-1},z_{2l+1}\right)\right)h\left(p,z_{2p-1}\right)\right)\right|.\label{eq:Sum to estimate}
\end{equation}
\[
\]

Adding and subtracting $\ind_{\left\{ x-u\right\} }\left(S_{k_{2l}}\right)\ind_{\left\{ v-y\right\} }\left(S_{k_{2l}+1}\right)$
from $h\left(l,u,v\right)$ we have
\begin{eqnarray*}
h\left(l,u,v\right) & = & \ind_{\left\{ x-u\right\} }\left(S_{k_{2l}}\right)\left(\ind_{\left\{ v-x\right\} }\left(S_{k_{2l+1}}\right)-\ind_{\left\{ v-y\right\} }\left(S_{k_{2l}+1}\right)\right)\\
 &  & +\left(\ind_{\left\{ x-u\right\} }\left(S_{k_{2l}}\right)-\ind_{\left\{ y-u\right\} }\left(S_{k_{2l}}\right)\right)\ind_{\left\{ v-y\right\} }\left(S_{2_{kl}+1}\right).
\end{eqnarray*}

At this point we use the inversion formula for the Fourier transform
to estimate (\ref{eq:Sum to estimate}). To do this, let $\bar{t}=\left(t_{1},...,t_{2p}\right)$
and write 
\begin{eqnarray*}
\tilde{h}_{1}\left(l,u,v\right) & = & e^{it_{2l}\left(x-u\right)}\left(e^{-it_{2l+1}\left(z_{2l+1}-x\right)}-e^{-it_{2l+1}\left(z_{2l+1}-y\right)}\right)\\
\tilde{h}_{2}\left(l,u,v\right) & = & \left(e^{it_{2l}\left(x-u\right)}-e^{it_{2l}\left(y-u\right)}\right)e^{it_{2l+1}\left(v-y\right)}\\
\tilde{h}\left(l,u\right) & = & e^{it_{2l}\left(x-u\right)}-e^{it_{2l}\left(y-u\right)}.
\end{eqnarray*}
 Fix $z_{1},z_{3},...,z_{p}$. Then by the inversion formula\\
$\qquad\qquad m\left(\ind_{\left\{ z_{1}\right\} }\left(S_{k_{1}}\right)\left(\prod_{l=1}^{p-1}h\left(l,z_{2l-1},z_{2l+1}\right)\right)h\left(p,z_{2p-1}\right)\right)=$
\[
\Re\int\limits _{\left[-\pi,\pi\right]^{2p}}m\left(e^{i\left\langle t,\bar{S}_{\bar{k}}\right\rangle }\right)e^{it_{1}z_{1}}\left(\prod_{l=1}^{p-1}\tilde{h}_{1}\left(l,z_{2l-1},z_{2l+1}\right)\right)\tilde{h}\left(p,z_{2p-1}\right)dt_{1}...dt_{2p}
\]
\begin{eqnarray*}
\Re\int\limits _{\left[-\pi,\pi\right]^{2p}}m\left(e^{i\left\langle t,\bar{S}_{\bar{k}}\right\rangle }\right)e^{it_{1}z_{1}}\left(\prod_{l=1}^{p-1}\tilde{h}_{1}\left(l,z_{2l-1},z_{2l+1}\right)\right)\tilde{h}\left(p,z_{2p-1}\right)dt_{1}...dt_{2p}
\end{eqnarray*}
The next proposition completes the proof: 
\begin{prop}
Let $x,y\in\bbZ$, $\bar{z}=\left(z_{1},...,z_{p}\right),$ $\bar{w}=\left(w_{1},...,w_{p}\right)$
be two vectors with integer coordinates such that $z_{i}-w_{i}=x-y$
and let $\bar{k}$ be a $p$-tuple of nonnegative integers. Also,
let $\bar{\xi}$ be a vector with the $i$-th coordinate being equal
to either $e^{it_{i}z_{i}}-e^{it_{i}w_{i}}$ or $e^{it_{i}z_{i}}$.
Denote by $J$ the set of coordinates $1\leq i\leq p$ such the $\xi_{i}=e^{it_{i}z_{i}}-e^{it_{i}w_{i}}$
and by $\bar{J}$ the set of coordinates $1\leq i\leq p$ with $\xi_{i}=e^{itz_{i}}$.
Then 
\[
\sum_{1\leq k_{1}\leq...\leq k_{p}\leq n}\left|\Re\int\limits _{\left[-\pi,\pi\right]^{p}}m\left(e^{i\left\langle t,\bar{S}_{\bar{k}}\right\rangle }\right)\prod_{l=1}^{p}\xi_{l}\, dt_{1}...dt_{p}\right|\leq C_{p}\left|x-y\right|^{\#J}\sqrt{n}^{\#\bar{J}}+\sqrt{n}^{\#\bar{J}+1}\left|x-y\right|^{\#J-2}\log n+\sqrt{n}^{p-2}\left(\log n\right)^{2}
\]
where $C_{p}$ is a contant. \end{prop}
\begin{proof}
We may assume that $\left|x-y\right|\leq\sqrt{n}$ (since certainly
$\left|x-y\right|$ is bounded by constant times $\sqrt{n}$). By
definition of the characteristic function operator 
\[
m\left(e^{i\left\langle t,\bar{S}_{\bar{k}}\right\rangle }\right)=m\left(P_{t_{p}}^{k_{p}}P_{t_{p-1}}^{k_{p-1}}...P_{t_{1}}^{k_{1}}\ind\right).
\]
By (\ref{eq: Taylor's decomposition for characteristic fn op}) there
exist a $0<\delta<\pi$ such that $P_{t}=\lambda_{t}m+\lambda_{t}\eta_{t}+N_{t}$
where $\lambda_{t}\leq1-ct^{2}$ for some positive constant $c$,
$\left\Vert \eta_{t}\right\Vert \leq c\left|t\right|$ and the spectral
radius of $N_{t}$ satisfies $\rho\left(N_{t}\right)\leq q<1$ for
all $t\in\left(-\delta,\delta\right)$. We denote $C_{\delta}=\left(-\delta,\delta\right)$
and $\bar{C}_{\delta}=\left[-\pi,\pi\right]\setminus\left(-\delta,\delta\right)$.
Thus, 
\begin{eqnarray}
\Re\int_{\left[-\pi,\pi\right]^{p}}m\left(e^{i\left\langle t,\bar{S}_{\bar{k}}\right\rangle }\right)\prod_{l=1}^{p}\xi_{l}\, dt_{1}...dt_{p} & = & \Re\int\limits _{\left[-\pi,\pi\right]^{p-1}}\int\limits _{C_{\delta}}\lambda_{t_{p}}^{k_{p}}m\left(P_{t_{p-1}}^{k_{p-1}}...P_{t_{1}}^{k_{1}}\ind\right)\prod_{l=1}^{p}\xi_{l}\, dt_{p}...dt_{1}\label{eq: Char. Fn. Inequality}\\
 &  & +\Re\int\limits _{\left[-\pi,\pi\right]^{p-1}}\int\limits _{C_{\delta}}m\left(\lambda_{t_{p}}^{k_{p}}\eta_{t_{p}}P_{t_{p-1}}^{k_{p-1}}...P_{t_{1}}^{k_{1}}\ind\right)\prod_{l=1}^{p}\xi_{l}\, dt_{p}...dt_{1}\nonumber \\
 &  & +\Re\int\limits _{\left[-\pi,\pi\right]^{p-1}}\int\limits _{C_{\delta}}m\left(N_{t_{p}}P_{t_{p-1}}^{k_{p-1}}...P_{t_{1}}^{k_{1}}\ind\right)\prod_{l=1}^{p}\xi_{l}\, dt_{p}...dt_{1}\nonumber \\
 &  & +\Re\int\limits _{\left[-\pi,\pi\right]^{p-1}}\int\limits _{\bar{C}_{\delta}}m\left(P_{t_{p}}^{k_{p}}P_{t_{p-1}}^{k_{p-1}}...P_{t_{1}}^{k_{1}}\ind\right)\prod_{l=1}^{p}\xi_{l}\, dt_{p}...dt_{1}\nonumber 
\end{eqnarray}
We handle each of the terms on the right hand side separately. Since
$\int\limits _{\left[-\pi,\pi\right]^{p-1}}m\left(P_{t_{p-1}}^{k_{p-1}}...P_{t_{1}}^{k_{1}}\ind\right)\prod_{l=1}^{p-1}\xi_{l}\, dt_{p-1}...dt_{1}$
is a difference of inverse Fourier transforms and therefore real,
\begin{eqnarray*}
\Re\int\limits _{\left[-\pi,\pi\right]^{p-1}}\int\limits _{C_{\delta}}\lambda_{t_{p}}^{k_{p}}m\left(P_{t_{p-1}}^{k_{p-1}}...P_{t_{1}}^{k_{1}}\ind\right)\prod_{l=1}^{p}\xi_{l}\, dt_{p}...dt_{1} & = & \left(\Re\int\limits _{\left[-\pi,\pi\right]^{p-1}}m\left(P_{t_{p-1}}^{k_{p-1}}...P_{t_{1}}^{k_{1}}\ind\right)\prod_{l=1}^{p-1}\xi_{l}\, dt_{p-1}...dt_{1}\right)\\
 &  & \cdot\left(\Re\int_{C_{\delta}}\lambda_{t_{p}}^{k_{p}}\xi_{p}dt_{p}\right)
\end{eqnarray*}
If $p\in J$, by the proof of the potential kernel estimate (proposition
\ref{prop: Potential Kernel}) $\sum_{1\leq k_{p}\leq n}\left|\Re\int_{C_{\delta}}\lambda_{t_{p}}^{k_{p}}\xi_{p}dt_{p}\right|\leq C_{1}\left|x-y\right|$.
Otherwise by remark \ref{rem: Norm Boundness} this term is bounded
by $\frac{C_{1}}{\sqrt{k_{p}}}$ and $\sum_{1\leq k_{p}\leq n}\frac{C_{1}}{\sqrt{k_{p}}}\leq C_{2}\sqrt{n}$.
To estimate 
\[
\Re\int\limits _{\left[-\pi,\pi\right]^{p-1}}m\left(P_{t_{p-1}}^{k_{p-1}}...P_{t_{1}}^{k_{1}}\ind\right)\prod_{l=1}^{p-1}\xi_{l}\, dt_{p-1}...dt_{1}
\]
 we may use the induction hypothesis. Combining the two estimates
we obtain
\[
\sum_{1\leq k_{1},...,k_{p}\leq n}\left|\Re\int\limits _{\left[-\pi,\pi\right]^{p-1}}\int\limits _{C_{\delta}}\lambda_{t_{p}}^{k_{p}}m\left(P_{t_{p-1}}^{k_{p-1}}...P_{t_{1}}^{k_{1}}\ind\right)\prod_{l=1}^{p}\xi_{l}\, dt_{p}...dt_{1}\right|
\]
\[
\leq C_{2}\left|x-y\right|^{\#J}\sqrt{n}^{\#\bar{J}}+\sqrt{n}^{\#\bar{J}+1}\left|x-y\right|^{\#J-2}\log n+\sqrt{n}^{p-2}\left(\log n\right)^{2}.
\]
We now turn to the second term in \ref{eq: Char. Fn. Inequality}.
Expanding one more term in the integral we get\\
$\qquad\qquad\Re\int\limits _{\left[-\pi,\pi\right]^{p-1}}\int\limits _{C_{\delta}}m\left(\lambda_{t_{p}}^{k_{p}}\eta_{t_{p}}P_{t_{p-1}}^{k_{p-1}}...P_{t_{1}}^{k_{1}}\ind\right)\prod_{l=1}^{p}\xi_{l}\, dt_{p}...dt_{1}$
\begin{eqnarray}
=\Re\int\limits _{\left[-\pi,\pi\right]^{p-2}}\int\limits _{C_{\delta}\times C_{\delta}}m\left(\lambda_{t_{p}}^{k_{p}}\eta_{t_{p}}\lambda_{t_{p-1}}^{k_{p-1}}m\left(P_{t_{p-2}}^{k_{p-2}}...P_{t_{1}}^{k_{1}}\ind\right)\right)\prod_{l=1}^{p}\xi_{l}\, dt_{p}...dt_{1}\label{eq: Char. Fn. Ineq. 2}\\
+\Re\int\limits _{\left[-\pi,\pi\right]^{p-2}}\int\limits _{C_{\delta}\times C_{\delta}}m\left(\lambda_{t_{p}}^{k_{p}}\eta_{t_{p}}\lambda_{t_{p-1}}^{k_{p-1}}\eta_{t_{p-1}}P_{t_{p-2}}^{k_{p-2}}...P_{t_{1}}^{k_{1}}\ind\right)\prod_{l=1}^{p}\xi_{l}\, dt_{p}...dt_{1}\nonumber \\
+\Re\int_{\left[-\pi,\pi\right]^{p-2}}\int\limits _{C_{\delta}\times C_{\delta}}m\left(\lambda_{t_{p}}^{k_{p}}\eta_{t_{p}}N_{t_{p-1}}^{k_{p-1}}P_{t_{p-2}}^{k_{p-2}}...P_{t_{1}}^{k_{1}}\ind\right)\prod_{l=1}^{p}\xi_{l}\, dt_{p}...dt_{1}\nonumber \\
+\Re\int_{\left[-\pi,\pi\right]^{p-2}}\int\limits _{\bar{C}_{\delta}}\int\limits _{C_{\delta}}m\left(\lambda_{t_{p}}\eta_{t_{p}}P_{t_{p-1}}^{k_{p-1}}...P_{t_{1}}^{k_{1}}\ind\right)\prod_{l=1}^{p}\xi_{l}\, dt_{p}...dt_{1}.\nonumber 
\end{eqnarray}
Since $\left\Vert \eta_{t}\right\Vert \leq c\left|t\right|$, by remark
\ref{rem: Norm Boundness} we get that $\int_{C_{\delta}}\left\Vert \lambda_{t_{p}}^{k_{p}}\eta_{t_{p}}\right\Vert dt_{p}\leq\frac{C_{1}}{k_{p}}$,
and hence $\sum_{k_{p}=1}^{n}\left|\int_{C_{\delta}}\lambda_{t_{p}}^{k_{p}}\eta_{t_{p}}\right|dt_{p}\leq\log n$.
Thus, using also $\int_{C_{\delta}}\left|\lambda_{t_{p}}^{k_{p-1}}\right|dt_{p}\leq\frac{C_{1}}{\sqrt{k_{p-1}}}$\\
$\sum_{1\leq k_{1},...,k_{p}\leq n}\left|\Re\int\limits _{\left[-\pi,\pi\right]^{p-2}}\int\limits _{C_{\delta}\times C_{\delta}}m\left(\lambda_{t_{p}}^{k_{p}}\eta_{t_{p}}\lambda_{t_{p-1}}^{k_{p-1}}m\left(P_{t_{p-2}}^{k_{p-2}}...P_{t_{1}}^{k_{1}}\ind\right)\right)\prod_{l=1}^{p}\xi_{l}\, dt_{p}...dt_{1}\right|$
\begin{eqnarray*}
\\
 & \leq & \sqrt{n}\log n\sum_{1\leq k_{1},...,k_{p-2}\leq n}\left|\int\limits _{\left[-\pi,\pi\right]^{p-2}}m\left(P_{t_{p-2}}^{k_{p-2}}...P_{t_{1}}^{k_{1}}\ind\right)\prod_{l=1}^{p}\xi_{i-2}\, dt_{p}...dt_{1}\right|\\
 & \leq & C_{2}\left(\sqrt{n}^{\#\bar{J}+1}\left|x-y\right|^{\#J-2}\log n+\sqrt{n}^{p-2}\left(\log n\right)^{2}\right).
\end{eqnarray*}
Using the same method and $\int_{\left[-\pi,\pi\right]}\left\Vert P_{t}^{k}\right\Vert dt\leq\frac{C}{\sqrt{n}}$
we obtain 
\[
\sum_{1\leq k_{1},...,k_{p-2}\leq n}\left|\Re\int\limits _{\left[-\pi,\pi\right]^{p-2}}\int\limits _{C_{\delta}\times C_{\delta}}m\left(\lambda_{t_{p}}^{k_{p}}\eta_{t_{p}}\lambda_{t_{p-1}}^{k_{p-1}}\eta_{t_{p-1}}P_{t_{p-2}}^{k_{p-2}}...P_{t_{1}}^{k_{1}}\ind\right)\prod_{l=1}^{p}\xi_{l}\, dt_{p}...dt_{1}\right|\leq C\left(\sqrt{n}^{p-2}\left(\log n\right)^{2}\right).
\]
Keeping in mind that $\left\Vert P_{t}^{n}\right\Vert $ for $t\in\bar{C}_{\delta}$
and $\left\Vert N_{t}^{n}\right\Vert $ for $t\in C_{\delta}$ uniformly
tend to $0$ with an exponential rate we obtain the bound $C\left(\sqrt{n}^{p-2}\left(\log n\right)^{2}\right)$
for the third and fourth term on the righ hand side of \ref{eq: Char. Fn. Ineq. 2}
(actually we obtain a better bound, but we do not use it). 

Combination of the estimates above proves the result.

\[
\]
\[
\]

\end{proof}

\section{\label{sec:Tightness}Tightness of $l_{n}$ in $D$.}

A sequence $\left\{ X_{n}\right\} $ of random variables taking values
in a standard Borel Space $\left(X,\calB\right)$ is called tight
if for every $\epsilon>0$ there exists a compact $K\subset X$ such
that for every $n\in\bbN$,
\[
P_{n}(K)>1-\epsilon,
\]
where $P_{n}$ denotes the distribution of $X_{n}$ . By Prokhorov's
Theorem relative compactness of $t_{n}(x)$ in $D$ is equivalent
to tightness. Therefore we are interested in characterizing tightness
in $D$.

For $x\left(t\right)$ in $D_{\left[-h,h\right]},$$T\subseteq\left[-h,h\right]$
set 
\[
\omega_{x}\left(T\right)=\sup_{s,t\in T}\left|x(s)-x(t)\right|
\]
and 

\[
\omega_{x}(\delta):=\sup_{\left|s-t\right|<\delta}\left|x\left(s\right)-x\left(t\right)\right|.
\]
$\omega_{x}\left(\delta\right)$ is called the modulus of continuity
of $x$. Due to the Arzela - Ascoli theorem, it plays a central role
in characterizing tightness in the space $C\left[-h,h\right]$ of
continuous functions on $\left[-h,h\right]$, with a Borel $\sigma$-algebra
generated by the topology of uniform convergence. 

The function that plays in $D\left[-h,h\right]$ the role that the
modulus of continuity plays in $C\left[-h,h\right]$ is defined by
\[
\omega_{x}'(\delta)=\inf_{\left\{ t_{i}\right\} }\max_{1\leq i\leq v}\omega(\left[t_{i},t_{i+1})\right),
\]
where $\left\{ t_{i}\right\} $ denotes a $\delta$ sparse partition
of $[-h,h]$, i.e. $\left\{ t_{i}\right\} $ is a partition $-h=t_{1}<t_{2}<...<t_{v+1}=h$
such that ${\displaystyle \min_{1\leq i\leq v}\left|t_{i+1}-t_{i}\right|>\delta}$.
It is easy to check that if $\frac{1}{2}>\delta>0$, and $h\geq1,$
\[
\omega'_{x}(\delta)\leq\omega_{x}\left(2\delta\right).
\]
For details see \cite[Sections 12 and 13]{Bil}. The next theorem
is a characterization of tightness in the space $D$. 
\begin{thm}
\cite[Lemma 3, p.173]{Bil} (1)The sequence $l_{n}$ is tight in $D$
if and only if its restriction to $\left[-h,h\right]$ is tight in
$D_{\left[-h,h\right]}$ for every $h\in\bbR_{+}$.

(2) The sequence $l_{n}$ is tight in $D_{\left[-h,h\right]}$ if
and only if the following two conditions hold:

(i) $\forall x\in\left[-h,h\right],\;\lim\limits _{a\rightarrow\infty}\limsup\limits _{n\rightarrow\infty}m\left[\left|l{}_{n}\left(x\right)\right|\geq a\right]=0.$

(ii) $\forall\epsilon>0,\,\lim\limits _{\delta\rightarrow0}\limsup\limits _{n\rightarrow\infty}m\left[\omega_{l_{n}}^{'}\left(\delta\right)\geq\epsilon\right]=0.$ \end{thm}
\begin{rem}
See \cite[Thm. 13.2]{Bil} and the Corollary that follows. \label{rem: Norm Boundness}Conditions
$\left(i\right)$ and $\left(ii\right)$ of the previous theorem imply
that 
\[
\lim\limits _{a\rightarrow\infty}\limsup\limits _{n\rightarrow\infty}m\left[\sup_{x\in[-m,m]}\left|l{}_{n}\left(x\right)\right|\geq a\right]=0.
\]
\end{rem}
\begin{prop}
\label{pro:The-sequence is tight}The sequence $\left\{ l_{n}\right\} _{n=1}^{\infty}$
is tight. \end{prop}
\begin{proof}
We prove that condition $2(i)$ holds. 

Fix $\epsilon>0,$ $x\in\RR.$ Since the Brownian Motion $\omega(t)$
satisfies 
\[
\lim_{M\to\infty}\mbox{\ensuremath{\PP}}\left(\sup_{t\in[0,1]}\left|\omega(t)\right|>M\right)=0
\]
and $\omega_{n}$ converges in distribution to $\omega$ there are
$M,n_{0}$ such that for all $n>n_{0},$

\[
m\left(\sup_{t\in[0,1]}\left|W_{n}(t)\right|>M\right)<\epsilon.
\]

By definition of $t_{n}\left(x\right),$ it follows that if $\left|x\right|>M,$
$n>n_{0},$ 
\[
m\left(\left|l_{n}\left(x\right)\right|>0\right)<\epsilon.
\]

Now, if $\left|x\right|\leq M$, by proposition \ref{pro:The-sequence is tight},

\begin{eqnarray*}
m\left(\left|l_{n}\left(x\right)\right|>a\right) & \leq & m\left(\left|l_{n}\left(M+1\right)\right|>0\right)+m\left(\left|l_{n}\left(x\right)-l_{n}\left(M+1\right)\right|>a\right)\\
 & \leq & \epsilon+\frac{4C\cdot\left(M+1\right)^{2}}{a^{4}}
\end{eqnarray*}

and the last expression can be made less then $2\epsilon$ for sufficiently
large $a$. 

To prove condition $2(ii)$ WLOG we may assume that $m\geq1$. Since
$\omega'_{x}(\delta)\leq\omega_{x}\left(2\delta\right)$, it is sufficient
to prove that the stronger condition
\begin{equation}
\forall\epsilon>0.\ \lim\limits _{\delta\rightarrow0}\limsup\limits _{n\rightarrow\infty}m\left[\sup_{x,y\in[-h.h];\left|x-y\right|<\delta}\left|l_{n}(x)-l_{n}(y)\right|\geq\epsilon\right]=0\label{eq: sufficient condition for tightness}
\end{equation}
holds. 

Let $\epsilon>0$, $1<\alpha<2$. By proposition \ref{pro:The-sequence is tight}
there exists $C>0$ such that for all $x,y:\ \frac{1}{\sqrt{n}}\leq|x-y|\leq1$
\begin{equation}
\PP^{\mu}\left(\left|l_{n}(x)-l_{n}(y)\right|>\epsilon\right)\leq\frac{C}{\epsilon^{6}}\left|x-y\right|^{\alpha}.\label{eq:sufficient condition for tightness 2}
\end{equation}
Let $\delta>0$ and $n>\delta^{-2}$, notice that $l_{n}$ is constant
on segments of the form $\left[\frac{j}{\sqrt{n}},\frac{j+1}{\sqrt{n}}\right)$,
hence 
\[
m\left(\sup_{x,y\in[-h.h];\left|x-y\right|<\delta}\left|l_{n}(x)-l_{n}(y)\right|\geq4\epsilon\right)\leq\sum\limits _{|k\delta|\leq h}m\left(\sup\limits _{k\delta\sqrt{n}\leq j\leq\left(k+1\right)\delta\sqrt{n}}\left|l_{n}\left(k\delta\right)-l_{n}\left(\frac{j}{\sqrt{n}}\right)\right|\geq\epsilon\right).
\]
By \cite[Theorem 10.2]{Bil} it follows from (\ref{eq:sufficient condition for tightness 2})
that there exists $C_{2}>0$ such that 
\begin{equation}
m\left(\sup\limits _{k\delta\sqrt{n}\leq j\leq\left(k+1\right)\delta\sqrt{n}}\left|t_{n}\left(k\delta\right)-t_{n}\left(\frac{j}{\sqrt{n}}\right)\right|\geq\epsilon\right)\leq\frac{C_{2}}{\epsilon^{6}}\delta^{\alpha}.\label{eq:Billingsley2}
\end{equation}

Therefore, 
\[
m\left(\sup_{x,y\in[-h.h];\left|x-y\right|<\delta}\left|l_{n}(x)-l_{n}(y)\right|\geq4\epsilon\right)\leq\frac{2C_{2}m}{\epsilon^{6}}\delta^{\alpha-1}\xrightarrow[\delta\to0]{}0.
\]

\end{proof}

\section{\label{Sec:Identifying the Only Possible Limit Point}Identifying
$l$ as the Limit of A Convergent Subsequence of $\left\{ l_{n}\right\} _{n\in\protect\NN}$.}
\begin{prop}
\label{pro:Identification ot the limit}Assume that the sequence $\left\{ X_{n}\right\} $
satisfies the assumptions of theorem \ref{thm:Main Theorem}. Let
$l_{n_{k}}$ be some subsequence of $l_{n}$ that converges in distribution
to some limit $q.$ Then $q\deq l$. \end{prop}
\begin{proof}
Let $G_{k}=\left\{ a_{1},b_{1},...,a_{k},b_{k}:\, a_{i}<b_{i},\: i=1,...,k\right\} $.
For $g\in G_{k}$ define the transformation $\pi_{g}:D\rightarrow\bbR^{k}$
by $\pi_{g}\left(l\right)=\left(\int_{a_{1}}^{b_{1}}l\left(x\right)\, dx,...\int_{a_{n}}^{b_{n}}l\left(x\right)\, dx\right)$.
Clearly, 
\[
\mathcal{G}=\bigcup_{k=1}^{\infty}\left\{ \pi_{g}^{-1}\left(\left[c_{1},d_{1}\right)\times...\times\left[c_{k},d_{k}\right)\right):\, g\in G_{k},\, c_{i},d_{i}\in\bbR,\, c_{i}<d_{i},\, i=1,...,n\right\} 
\]
 is a $\pi$-system, i.e. closed under finite intersections. Moreover,
$\mathcal{G}$ generates the Borel $\sigma$-algebra of $D$. It follows
that if $\pi_{g}\left(q\right)=\pi_{g}\left(l\right)$ for every $g\in\bigcup_{k=1}^{\infty}G_{k}$
then $q\overset{d}{=}l$. Hence, to prove the theorem it is enough
to show that if $l_{n_{k}}\overset{d}{\longrightarrow}q$ then $\pi_{g}\left(l_{n_{k}}\right)\overset{d}{\longrightarrow}\pi_{g}\left(l\right)$
for every $g\in\bigcup_{k=1}^{\infty}G_{k}$. $ $To this purpose
we first prove that for $g=\left[a_{1},b_{1}\right)\times...\times\left[a_{k},b_{k}\right)$,
\begin{equation}
\pi_{g}\left(l_{n}\right)-\left(\int_{0}^{1}\One_{[a_{1},b_{1})}\left(l_{n}\right)dt,...,\int_{0}^{1}\One_{[a_{k},b_{k})}\left(l_{n}\right)dt\right)\overset{d}{\longrightarrow}0.\label{eq: Convergence in dist. 1}
\end{equation}
Then we prove that 
\begin{equation}
\left(\int_{0}^{1}\One_{[a_{1},b_{1})}\left(l_{n}\right)dt,...,\int_{0}^{1}\One_{[a_{k},b_{k})}\left(l_{n}\right)dt\right)\overset{d}{\longrightarrow}\left(\int_{0}^{1}\One_{[a_{1},b_{1})}\left(l\right)dt,...,\int_{0}^{1}\One_{[a_{k},b_{k})}\left(l\right)dt\right).\label{eq: Convergence in distribution 2}
\end{equation}
\ref{eq: Convergence in dist. 1} and \ref{eq: Convergence in distribution 2}
imply that $\pi_{g}\left(l_{n_{k}}\right)\overset{d}{\longrightarrow}\pi_{g}\left(l\right)$,
thus proving the proposition (see Billingsley).

We now prove \ref{eq: Convergence in dist. 1}. By straightforward
calculations using definitions, we have 
\begin{equation}
\left|\int_{0}^{1}\One_{[a,b)}\left(\omega_{n}(t)\right)dt-\int_{a}^{b}l_{n}(x)dx\right|\leq\int\limits _{\frac{\left\lfloor \sqrt{n}a\right\rfloor }{\sqrt{n}}}^{\frac{\left\lfloor \sqrt{n}a\right\rfloor +1}{\sqrt{n}}}l_{n}\left(x\right)dx+\int\limits _{\frac{\left\lfloor \sqrt{n}b\right\rfloor }{\sqrt{n}}}^{\frac{\left\lfloor \sqrt{n}b\right\rfloor +1}{\sqrt{n}}}l_{n}\left(x\right)dx.\label{eq:Difference in densities}
\end{equation}
Now, 
\begin{eqnarray*}
m\left(\left|\int\limits _{\frac{\left\lfloor \sqrt{n}a\right\rfloor }{\sqrt{n}}}^{\frac{\left\lfloor \sqrt{n}a\right\rfloor +1}{\sqrt{n}}}l_{n}\left(x\right)dx\right|>\epsilon\right) & \leq & m\left(\sup_{x\in\left[a-1,b+1\right]}\left|l_{n}\left(x\right)\right|>M\right)\\
 & + & m\left(\left|\int\limits _{\frac{\left\lfloor \sqrt{n}a\right\rfloor }{\sqrt{n}}}^{\frac{\left\lfloor \sqrt{n}a\right\rfloor +1}{\sqrt{n}}}l_{n}\left(x\right)dx\right|>\epsilon,\sup_{x\in\left[a-1,b+1\right]}\left|l_{n}\left(x\right)\right|\leq M\right).
\end{eqnarray*}
The second summand on the right side of the above inequality tends
to $0$ since the integral is less than $\frac{M}{\sqrt{n}}$ .The
first summand is arbitrarily close to $0$ for $M,n$ large enough,
by Remark \ref{rem: Norm Boundness}. Same reasoning applied to both
summands of equation (\ref{eq:Difference in densities}) gives 
\begin{equation}
\left|\int_{0}^{1}\One_{[a,b)}\left(\omega_{n}(t)\right)dt-\int_{a}^{b}l_{n}(x)dx\right|\overset{d}{\longrightarrow}0.\label{eq: Converg. In Dist. 3}
\end{equation}
\ref{eq: Convergence in dist. 1} now follows from 
\[
m\left(\left\Vert \pi_{g}\left(l_{n}\right)-\left(\int_{0}^{1}\One_{[a_{1},b_{1})}\left(l_{n}\right)dt,...,\int_{0}^{1}\One_{[a_{k},b_{k})}\left(l_{n}\right)dt\right)\right\Vert >\epsilon\right)\leq\sum_{k=1}^{n}m\left(\left|\int_{0}^{1}\One_{[a_{k},b_{k})}\left(\omega_{n}(t)\right)dt-\int_{a_{k}}^{b_{k}}l_{n}(x)dx\right|>\frac{\epsilon}{k}\right)
\]
and \ref{eq: Converg. In Dist. 3}. 

We turn to the proof of \ref{eq: Convergence in distribution 2}.
Since $\omega_{n}\overset{d}{\longrightarrow}\omega$, it is enough
to show that the transformation $\omega\rightarrow\int_{0}^{1}\One_{[a,b)}\left(\omega\left(t\right)\right)dt$
is continuous in the Skorokhod topology on $D\left[0,1\right]$ at
almost all sample points of the Brownian motion $\omega$. This is
proved in \cite[Section 2]{KS}. \end{proof}


\begin{thebibliography}{10}
\bibitem[2]{AD}J. Aaronson and M. Denker. Local limit theorems for
partial sums of stationary sequences generated by Gibbs\textendash Markov
maps. Stoch.Dyn. 1(2),193\textendash 237 (2001).

\bibitem[3]{Bil}Billingsley, P. Convergence of Probability Measures,
second edition, Wiley, New York (1999).

\bibitem[4]{Bor}Borodin, A.N. On the asymptotic behavior of local
times of recurrent random walks with finite variance. Theory Probab.
Appl. 26, pp. 758\textendash 772 (1981).

\bibitem[7]{HeH}Hennion H., and Hervé, L. Limit theorems for Markov
chains and stochastic properties of dynamical systems by quasi-compactness.
Lecture Notes in Mathematics, 1766. Springer-Verlag, Berlin (2001). 

\bibitem[8]{JP}Jain N., Pruitt W. Asymptotic behavior of the local
time of a recurrent random walk. Ann. Probab. Volume 12, Number 1,
64-85 (1984).

\bibitem[10]{KS}Kesten, H. and Spitzer, F. A limit theorem related
to a new class of self similar processes. Z. Wahrsch. Verw. Gebiete
50, no.1, pp. 5-25 (1979).

\bibitem[11]{MP}Mörters, P., Peres, Y.: Brownian Motion. Cambridge
University Press, Cambridge (2009).\end{thebibliography}
\end{document}